\documentclass[12pt,a4paper]{article}
\usepackage{amssymb,amsmath,amsthm}
\usepackage[english]{babel}
\usepackage{t1enc}
\usepackage[latin2]{inputenc}
\usepackage{epsfig}
\usepackage{subfig}
\usepackage{comment}
\usepackage{epstopdf}
\usepackage{url}
\usepackage{hyperref}
\usepackage{cleveref}
\usepackage{enumitem}

\usepackage[margin=1.1in]{geometry}

\newtheorem{thm}{Theorem}
\newtheorem{manualtheoreminner}{Theorem}
\newenvironment{thm'}[1]{%
	\renewcommand\themanualtheoreminner{#1}%
	\manualtheoreminner
}{\endmanualtheoreminner}

\newtheorem{lem}[thm]{Lemma}
\newtheorem{prop}[thm]{Proposition}

\newtheorem{conj}[thm]{Conjecture}

\newtheorem{obs}[thm]{Observation}

\theoremstyle{remark}

\usepackage{indentfirst}
\def\cC{\mathcal{C}}

\makeatletter
\def\blfootnote{\gdef\@thefnmark{}\@footnotetext}
\makeatother

\begin{document}

\title{On the number of tangencies among $1$-intersecting $x$-monotone curves}
\author{
Eyal Ackerman\thanks{Department of Mathematics, Physics and Computer Science,
	University of Haifa at Oranim, 	Tivon 36006, Israel.}\and
Bal\'azs Keszegh\thanks{Alfréd Rényi Institute of Mathematics and ELTE Eötvös Loránd University, Budapest, Hungary. 
	Research supported by the J\'anos Bolyai Research Scholarship of the Hungarian Academy of Sciences, by the National Research, Development and Innovation Office -- NKFIH under the grant K 132696 and FK 132060, by the \'UNKP-22-5 New National Excellence Program of the Ministry for Innovation and Technology from the source of the National Research, Development and Innovation Fund and by the ERC Advanced Grant ``ERMiD''. This research has been implemented with the support provided by the Ministry of Innovation and Technology of Hungary from the National Research, Development and Innovation Fund, financed under the  ELTE TKP 2021-NKTA-62 funding scheme.}
}
\maketitle

\begin{abstract}
Let $\cC$ be a set of curves in the plane such that no three curves in $\cC$ intersect at a single point and every pair of curves in $\cC$ intersect at exactly one point which is either a crossing or a touching point.
J\'anos Pach conjectured that the number of pairs of curves in $\cC$ that touch each other is $O(|\cC|)$. We prove this conjecture for $x$-monotone curves.
\end{abstract}

\section{Introduction}

We study the number of tangencies within a family of $1$-intersecting $x$-monotone planar curves.
A planar curve is a \emph{Jordan arc}, that is, the image of an injective continuous function from a closed interval into $\mathbb{R}^2$. 
If no two points on a curve have the same $x$-coordinate, then the curve is \emph{$x$-monotone}.
We consider families of curves such that every pair of curves intersect at a finite number of points.
Such a family is called $t$-intersecting if every pair of curves intersects at at most $t$ points.
An intersection point $p$ of two curves is a \emph{crossing point} if there is a small disk $D$ centered at $p$ which contains no other intersection point of these curves, each curve intersects the boundary of $D$ at exactly two points and in the cyclic order of these four points no two consecutive points belong to the same curve.
If two curves intersect at exactly one point which is not a crossing point, then we say that they are \emph{touching} or \emph{tangent} at that point.

The number of \emph{tangencies} is the number of tangent pairs of curves.
If more than two curves are allowed to intersect at a common point, then every pair of curves might be tangent, e.g., for the graphs of the functions $x^{2i}$, $i=1,2,\ldots,n$, in the interval $[-1,1]$.
Therefore, we restrict our attention to families of curves in which no three curves intersect at a common point.
It is not hard to construct such a family of $n$ ($x$-monotone) $1$-intersecting curves with $\Omega(n^{4/3})$ tangencies based on a famous construction of Erd\H{o}s (see~\cite{pachbook}) of $n$ lines and $n$ points admitting that many point-line incidences (for an illustration consider Figure~\ref{fig:many-tangency-points} and trim every curve beyond its rightmost touching point).
J\'anos Pach~\cite{pachpc} conjectured that requiring every pair of curves to intersect (either at crossing or a tangency point) leads to significantly less tangencies.
 
\begin{conj}[\cite{pachpc}]\label{conj:Pach}
Let $\cC$ be a set of $n$ curves such that no three curves in $\cC$ intersect at a single point and every pair of curves in $\cC$ intersect at exactly one point which is either a crossing or a tangency point.
Then the number of tangencies among the curves in $\cC$ is $O(n)$.
\end{conj}

Gy\"orgyi, Hujter and Kisfaludi-Bak~\cite{GYORGYI201829} proved Conjecture~\ref{conj:Pach} for the special case where there are constantly many faces in the arrangement of $\cC$ that together contain all the endpoints of the curves.
In this paper we show that Conjecture~\ref{conj:Pach} also holds for $x$-monotone curves.

\begin{thm}\label{thm:pw-x-monotone-1-intersecting}
Let $\cC$ be a set of $n$ $x$-monotone curves such that no three curves in $\cC$ intersect at a single point and every pair of curves in $\cC$ intersect at exactly one point which is either a crossing or a tangency point.
Then the number of tangencies among the curves in $\cC$ is at most $1160 n$.
\end{thm}

We prove Theorem~\ref{thm:pw-x-monotone-1-intersecting} by considering two types of tangencies according to whether a tangency point is between two curves such that their projections on the $x$-axis are nested (i.e., one of them is a subset of the other) or non-nested. In each case we consider the tangencies graph whose vertices represent the curves and whose edges represent tangent pairs of curves. In the latter case we show that it is possible to disregard some constant proportion of the edges using the pigeonhole principle and the dual of Dilworth's Theorem and then order the remaining edges such that there is no long monotone increasing path with respect to this order. 
In the first case, we show that after disregarding some constant proportion of the edges the remaining edges induce a forest.

\paragraph{Related Work.} It follows from a result of Pach and Sharir~\cite{PS91} that $n$ $x$-monotone $1$-intersecting curves admit $O\left(n^{4/3}\left(\log n\right)^{2/3}\right)$ tangencies. Note that this bound almost matches the lower bound mentioned above.
It also follows from~\cite{PS91} that for bi-infinite $x$-monotone $1$-intersecting curves the maximum number of tangencies is $\Theta(n \log n)$.
P\'alv\"olgyi et nos~\cite{AKP23} showed that there are $O(n)$ tangencies among families of $n$ $1$-intersecting curves that can be partitioned into two sets such that all the curves within each set are pairwise disjoint.
Variations of this bipartite setting were also studied in~\cite{Ackerman2013,KP21,Treml}.

Pach, Rubin and Tardos~\cite{PRT16,PRT18} settled a long-standing conjecture of Richter and Thomassen~\cite{RT95} concerning the number of crossing points determined by pairwise intersecting curves.
In particular, they showed that in any set of curves admitting linearly many tangencies the number of crossing points is superlinear with respect to the number of tangencies.
This implies that for any fixed $t$ every set of $n$ $t$-intersecting curves admits $o(n^2)$ tangencies.
Salazar~\cite{S99} already pointed that out for such families which are also pairwise intersecting.
Better bounds for families of $t$-intersecting curves were found in~\cite{maya,KP21}.
Specifically, it follows from~\cite{KP21} that $n$ $1$-intersecting curves determine $O(n^{7/4})$ tangencies.

In many cases one can locally perturb curves such that tangencies become digons (faces of size two in the arrangement of the curves) and vice versa, although
one has to be careful since this might lead to more intersection points or a different notion of tangency (e.g., two curves may form many digons but can they touch in more than one point?).
There are quite a few problems in combinatorial geometry that can be phrased in terms of bounding the number of tangencies (or digons) among certain curves, see, e.g.,~\cite{agarwal}.
The most famous of which is the unit distance problem of Erd\H{o}s~\cite{erdos} which asks for the maximum number of unit distances among $n$ points in the plane. 
It is easy to see that this problem is equivalent to asking for the maximum number of tangencies among $n$ unit circles.

\section{Proof of Theorem~\ref{thm:pw-x-monotone-1-intersecting}}

Let $\cC$ be a set of $n$ $x$-monotone curves such that no three curves in $\cC$ intersect at a single point and every pair of curves in $\cC$ intersect at exactly one point which is either a crossing or a tangency point.
By slightly extending the curves if needed, we may assume that every intersection point of two curves is an interior point of both of them and that all the endpoints of the curves are distinct.

Let $p=(x_1,y_1)$ and $q=(x_2,y_2)$ be two points. We write $p <_x q$ if $x_1<x_2$ and we write $p <_y q$ if $y_1 < y_2$.
We mainly consider the order of points from left to right, so when we use terms like `before', `after' and `between' they should be understood in this sense.
For a curve $c\in \cC$ we denote by $L(c)$ and $R(c)$ the left and right endpoints of $c$, respectively.
If $p,q \in c$, then $c(p,q)$ denotes the part of $c$ between these two points.
We denote by $c(-,p)$ (resp., $c(p,+)$) the part of $c$ between $L(c)$ (resp., $R(c)$) and $p$.
For another curve $c' \in \cC$ we denote by $I(c,c')$ the intersection point of $c$ and $c'$.
We may also write, e.g., $c(c',q)$ instead of $c(I(c,c'),q)$

Suppose that an $x$-monotone curve $c_1$ lies \emph{above} another $x$-monotone curve $c_2$, that is, the two curves are non-crossing (but might be touching) and there is no vertical line $\ell$ such that $I(c_1,\ell) <_y I(c_2,\ell)$.
Assuming the endpoints of $c_1$ and $c_2$ are distinct there are four possible cases: (1)~$L(c_1) <_x L(c_2) <_x R(c_2) <_x R(c_1)$; (2)~$L(c_2) <_x L(c_1) <_x R(c_1) <_x R(c_2)$; (3)~$L(c_1) <_x L(c_2) <_x R(c_1) <_x R(c_2)$; and (4)~$L(c_2) <_x L(c_1) <_x R(c_2) <_x R(c_1)$.
We denote by $c_2 \prec_i c_1$ the relation that corresponds to case~$i$, for $i=1,2,3,4$.
It is not hard to see that each $\prec_i$ is a partial order.

\begin{prop}\label{prop:poset}
For every $i=1,2,3,4$ there are no three curves $c_1,c_2,c_3 \in \cC$ such that $c_1 \prec_i c_2 \prec_i c_3$.
\end{prop}

\begin{proof}
It is easy to see that if $c_1 \prec_i c_2 \prec_i c_3$ then $c_1$ and $c_3$ do not intersect.
See Figure~\ref{fig:c1-c2-c3} for an illustration of the case $i=4$.
\begin{figure}
	\centering
	\includegraphics[width= 6cm]{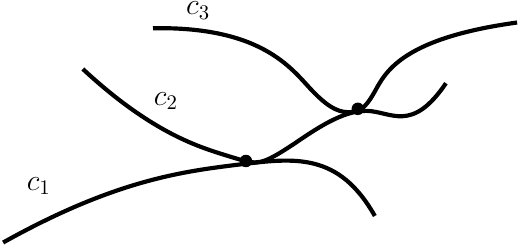}
	\caption{If $c_1 \prec_4 c_2 \prec_4 c_3$ then $c_1$ and $c_3$ do not intersect.}
	\label{fig:c1-c2-c3}		
\end{figure}
\end{proof}

We say that the tangency point of two touching curves $c_1,c_2 \in \cC$ is of \emph{Type} $i$ if $c_1 \prec_i c_2$ (see Figure~\ref{fig:types}).
\begin{figure}
	\centering
	\includegraphics[width= 12cm]{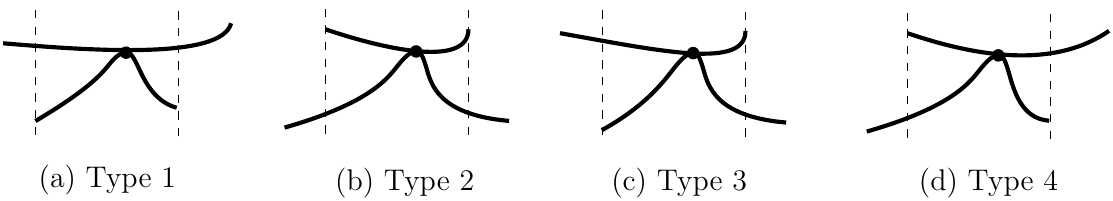}
	\caption{Types of tangency points.}
	\label{fig:types}		
\end{figure}
We will count separately tangency points of Types 1 and~2 and tangency points of Types~3 and~4.

\subsection{Bounding touching pairs of Type~1 or~2}

We first describe the main idea before going into details:
By symmetry it is enough to consider just tangency points of Type~2 which are to the right of a vertical line that intersects all the curves in $\cC$.
We will prove that there are linearly many such tangencies by showing that the tangencies graph that corresponds to such tangencies is a forest, if one ignores the rightmost touching point on every curve that touches another curve from above.

\begin{lem}\label{lem:nested}
There are at most $8n-4$ tangency points of Type~1 or~2.
\end{lem}

\begin{proof}
Since all the curves in $\cC$ are pairwise intersecting and $x$-monotone there is a vertical line $\ell$ that intersects all of them.
By slightly shifting $\ell$ if needed, we may assume that no two curves intersect $\ell$ at the same point.
We assume without loss of generality that at least half of all the tangency points of Types~1 and~2 are to the right of $\ell$, for otherwise we may reflect all the curves about $\ell$.
We may further assume that at least half of the tangency points of Types~1 and~2 to the right of $\ell$ are of Type~2, for otherwise we may reflect all the curves about the $x$-axis.
Henceforth, we consider only Type~2 tangency points to the right of $\ell$.

By Proposition~\ref{prop:poset} a curve cannot touch one curve from above and another curve from below at Type~2 tangency points. Thus, we may partition the curves into \emph{blue} curves and \emph{red} curves such that at every tangency point a blue curve touches a red curve from below (we ignore curves that contain no tangency points among the ones that we consider).

\begin{prop}\label{prop:blue-cross}
Every pair of blue curves cross each other.
\end{prop}

\begin{proof}
Suppose that a blue curve $b_1$ touches another blue curve $b_2$ from below (the tangency point may be of any type).
Since $b_2$ is a blue curve there is a red curve $r$ which it touches from below. 
Since the tangency point of $b_2$ and $r$ is of Type~2 it follows that $b_1$ and $r$ do not intersect.
\end{proof}

\begin{prop}\label{prop:blue-cross-between}
Let $r$ be a red curve and let $b_1$ and $b_2$ be two blue curves that touch $r$,
such that $I(r,b_1) <_x I(r,b_2)$.
Then $I(r,b_1) <_x I(b_1,b_2) <_x I(r,b_2)$.
\end{prop}

\begin{proof}
Since $L(b_i) <_x L(r) <_x R(r) <_x R(b_i)$, for $i=1,2$, it is easy to see that the blue curves cross at a point between their tangency points with $r$, see Figure~\ref{fig:blues-cross}.
\begin{figure}
	\centering
	\includegraphics[width= 8cm]{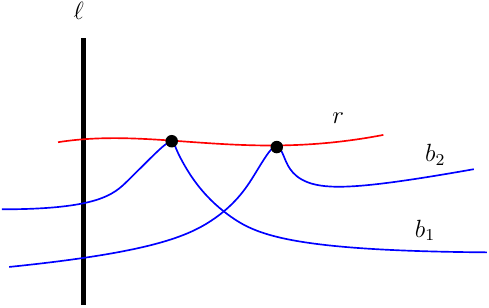}
	\caption{An illustration for the proof of Proposition~\ref{prop:blue-cross-between}: If two blue curves touch the same red curve, then they cross at a point between these two tangency points.}
	\label{fig:blues-cross}		
\end{figure}
\end{proof}

\begin{prop}\label{prop:b-above}
Let $b_1$ and $b_2$ be two blue curves both touching a red curve $r$.
Let $b$ be another blue curve such that $I(b_1,b_2)$ is between $I(b_1,b)$ and $I(b_2,b)$. Then $I(b_1,b_2)$ is below $b$.
\end{prop}

\begin{proof}
Observe that $r$ lies above the upper envelope of $b_1$ and $b_2$. 
Furthermore, since we consider Type~2 tangencies $L(b_i) <_x L(r) <_x R(r) <_x R(b_i)$, for $i=1,2$.
Therefore if $I(b_1,b_2)$ is above $b$ then $r$ and $b$ do not intersect, see Figure~\ref{fig:b-above}.
	\begin{figure}
	\centering
	\subfloat[]{\includegraphics[width= 7.5cm]{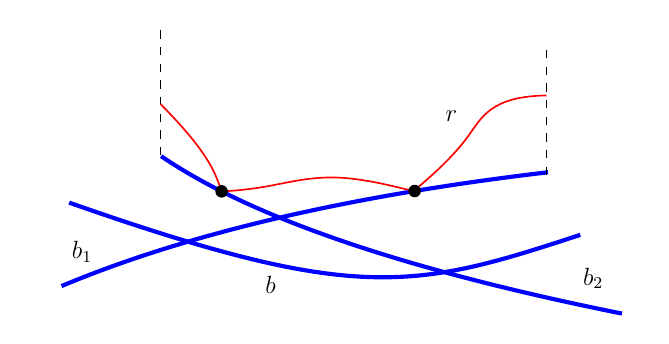}}
	\subfloat[]{\includegraphics[width= 7.5cm]{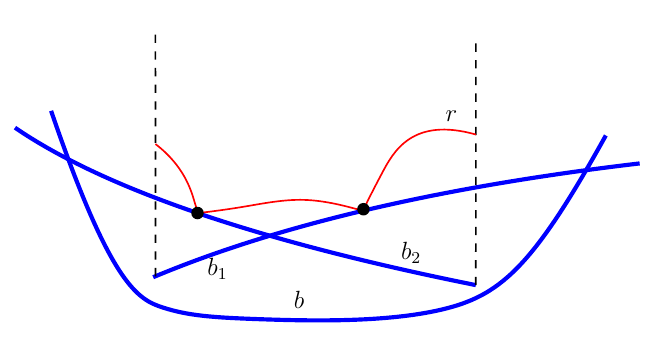}}
	\caption{An illustration for the proof of Proposition~\ref{prop:b-above}. If $I(b_1,b_2)$ is above $b$, then $r$ and $b$ do not intersect.}
	\label{fig:b-above}		
\end{figure}
\end{proof}

We proceed by marking the rightmost Type~2 tangency point on every red curve.
Clearly, at most $n$ tangency points are marked.
Henceforth, we consider only unmarked tangency points.

\begin{prop}\label{prop:red-cross-between}
Let $b$ be a blue curve and let $r_1$ and $r_2$ be two red curves that touch $b$,
such that $I(b,r_1) <_x I(b,r_2)$ (and both of these tangency points are unmarked).
Then $I(b,r_1) <_x I(r_1,r_2) <_x I(b,r_2)$.
\end{prop}

\begin{proof}
If $I(r_1,r_2) >_x I(b,r_2)$ then $r_2$ must intersect $r_1(b,r_2)$ since it intersects $\ell$, see Figure~\ref{fig:reds-cross2}.
\begin{figure}[t]
	\centering
	\subfloat[If $I(r_1,r_2) >_x I(b,r_2)$ then $r_2$ must also intersect $r_1(b,r_2)$.]{\includegraphics[width= 7cm]{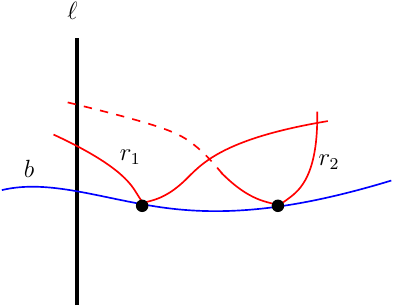}\label{fig:reds-cross2}}
	\hspace{5mm}
	\subfloat[If $I(r_1,r_2) <_x I(b,r_1$) then $b'$ and $r_1$ do not intersect.]{\includegraphics[width= 7cm]{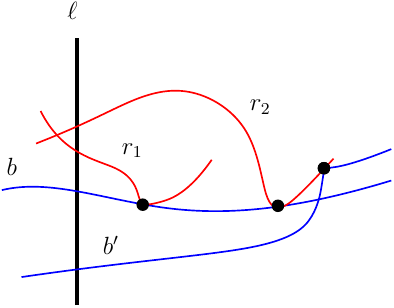}\label{fig:reds-cross}}
	\caption{Illustrations for the proof of Proposition~\ref{prop:red-cross-between}: if $r_1$ and $r_2$ touch $b$ then $I(r_1,r_2)$ is between these two tangency points.}
	\label{fig:red-cross-between}		
\end{figure}
Suppose now that $I(r_1,r_2) <_x I(b,r_1)$.
This implies that $R(r_1) <_x I(b,r_2)$ for otherwise $r_1$ and $r_2$ intersect also to the right of $I(b,r_1)$.
Since $I(b,r_2)$ is not the rightmost tangency point on $r_2$, there is a blue curve $b'$ that touches $r_2$ to the right of $I(b_,r_2)$ at a Type~2 tangency point.
However, it follows from Proposition~\ref{prop:blue-cross-between} that $I(b,r_2) <_x I(b,b') <_x I(b',r_2)$ which implies that $b'$ lies below $b$ to the left of $I(b,b')$.
Therefore, $b'$ and $r_1$ do not intersect (recall that $L(b) <_x L(r_1)$ and see Figure~\ref{fig:reds-cross}).
\end{proof}

Let $G$ be the (bipartite) \emph{tangencies graph} of the blue and red curves where edges correspond to unmarked tangency points between a red and a blue curve of Type~2 to the right of $\ell$.
We will show that $G$ is a forest and hence has at most $n-1$ edges.

\begin{prop}\label{prop:b_0-b_2}
Let $b_0-r_0-b_1-r_1-b_2$ be a path in $G$ such that $b_0$, $b_1$ and $b_2$ correspond to distinct blue curves and $r_0$ and $r_1$ correspond to distinct red curves. If $I(b_1,\ell) <_y I(b_0,\ell) <_y I(b_2,\ell)$, then
$b_0$ and $b_2$ intersect to the left of $\ell$.
\end{prop}

\begin{proof}
It follows from Proposition~\ref{prop:blue-cross-between}
that $I(b_0,b_1)$ and $I(b_1,b_2)$ are to the right of $\ell$.
Suppose for contradiction that $I(b_0,b_2)$ is to the right of $\ell$.
Recall that $b_2$ intersects $\ell$ above $b_0$ and $b_0$ intersects $\ell$ above $b_1$.
If, going from left to right, $b_2$ intersects first $b_1$ (necessarily to the right of $I(b_0,b_1)$) and then it intersects $b_0$, then $I(b_1,b_2)$ is above $b_0$ which contradicts Proposition~\ref{prop:b-above} (see Figure~\ref{fig:b_0-b_2-a}).
\begin{figure}[t]
	\centering
	\subfloat[$b_2$ intersects $b_1$ and then $b_0$.]{\includegraphics[width= 7cm]{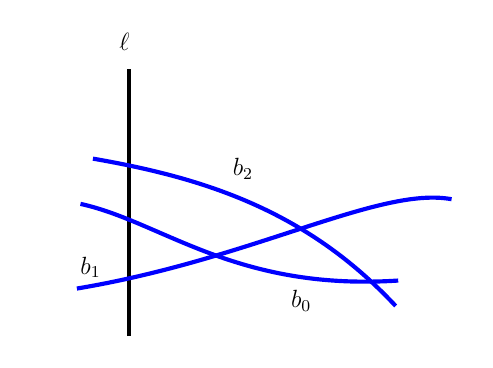}\label{fig:b_0-b_2-a}}
	\hspace{5mm}
	\subfloat[If $I(b_1,r_1) \in b_1(I(b_1,b_2),I(b_1,b_0))$ then $r_1$ must cross $b_0$ twice.]{\includegraphics[width= 7cm]{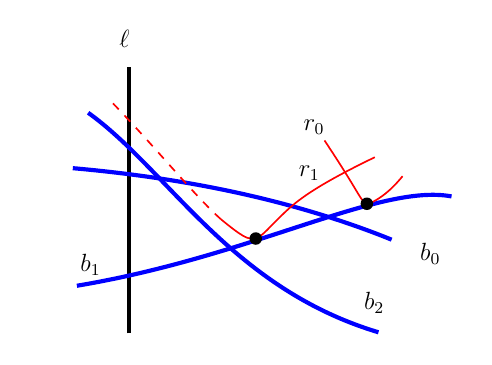}\label{fig:b_0-b_2-b}}
	\hspace{5mm}
	\subfloat[If $I(b_1,r_1) \in b_1(I(b_1,b_0),+)$, then $r_1$ touches $b_0$. If $I(b_1,r_0) <_x I(b_1,r_1)$, then $r_0$ must intersect $r_1$ twice or crosses $b_0$ to be able to intersect $\ell$.]{\includegraphics[width= 7cm]{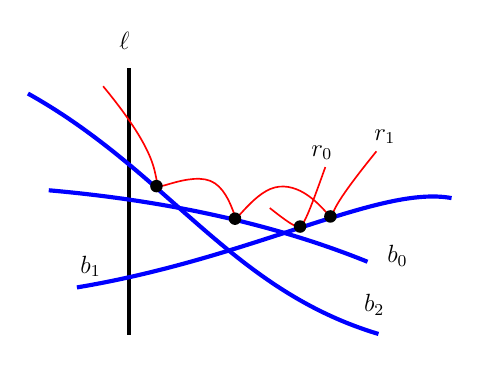}\label{fig:b_0-b_2-c}}
	\hspace{5mm}
	\subfloat[If $I(b_1,r_1) <_x I(b_1,r_0)$, then $r_0$ does not intersect $b_0$ to the right of $\ell$.]{\includegraphics[width= 7cm]{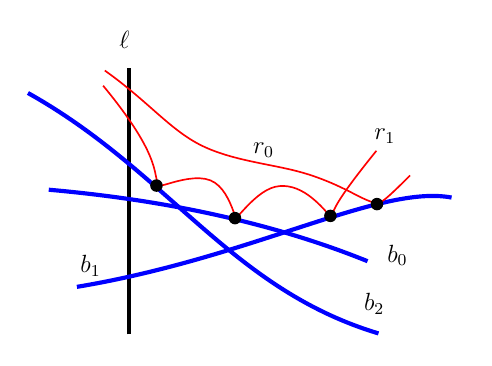}\label{fig:b_0-b_2-d}}
	\caption{Illustrations for the proof of Proposition~\ref{prop:b_0-b_2}. $I(b_0,b_2)$ is to the right of $\ell$.}
	\label{fig:b_0-b_2}		
\end{figure}

Therefore, $b_2$ intersects $b_0$ and then $b_1$ which implies that $b_1$ intersects $b_2$ and then $b_0$.
The curve $r_1$ lies above the upper envelope of $b_1$ and $b_2$, therefore it may touch $b_1$ at a point which is either in $b_1(b_2,b_0)$ or in $b_1(b_0,+)$.
Consider the first case.
It follows from Proposition~\ref{prop:red-cross-between} that $I(r_0,r_1)$ is between $I(b_1,r_1)$ and $I(b_1,r_0)$ and therefore $r_1$ must cross $b_0$ after it touches $b_1$, since $r_0$ is above the upper envelope of $b_0$ and $b_1$.
However, in this case $r_1$ must cross $b_0$ once more to the left of $I(b_1,r_1)$, since it also touches $b_2$ and therefore must lie above it (see Figure~\ref{fig:b_0-b_2-b}).

Consider now the case that $r_1$ touches $b_1$ at a point in $b_1(b_0,+)$.
Then $r_1$ must intersect $b_0(b_2,b_1)$ since this is the only part of $b_0$ which lies above the upper envelope of $b_1$ and $b_2$.
Since $r_1$ may not cross $b_2$ or intersect $b_0$ twice, it follows that $r_1$ must touch $b_0$ (see Figure~\ref{fig:b_0-b_2-c}).
Note that $r_0$ also touches $b_1$ at $b_1(b_0,+)$.
If $I(b_1,r_0)$ precedes $I(b_1,r_1)$ on $b_1$, then $r_0(b_1,+)$ must intersect $r_1$ by Proposition~\ref{prop:red-cross-between} (see Figure~\ref{fig:b_0-b_2-c}).
However, then $r_0$ must intersect $r_1$ once more or cross $b_0$ which is impossible.
If, on the other hand, $I(b_1,r_1)$ precedes $I(b_1,r_0)$ on $b_1$, then $r_1(b_1,+)$ must intersect $r_0$ by Proposition~\ref{prop:red-cross-between}.
However, then $r_0$ cannot touch $b_0$ to the right of $\ell$, see Figure~\ref{fig:b_0-b_2-d}.
\end{proof}

Suppose that $G$ contains a cycle and let $C=b_0-r_0-b_1-r_1-\ldots-b_k-r_k-b_0$ be a shortest cycle in $G$,
such that $b_i$ corresponds to a blue curve and $r_i$ corresponds to a red curve, for every $i=0,1,\ldots,k$.
We may assume without loss of generality that $b_1$ has the lowest intersection point with $\ell$ among the blue curves in $C$ and that $I(b_0,\ell) <_y I(b_2,\ell)$.

\begin{prop}\label{prop:structure}
For every $i \ge 1$ the curve $r_i$ intersects $\ell$ above $r_0$ and intersects $b_0(-,\ell)$, $r_0(b_0,+)$ and $b_1(b_0,+)$.
See Figure~\ref{fig:path-prop} for an illustration.
\begin{figure}[t]
	\centering
	\subfloat[]{\includegraphics[width= 7cm]{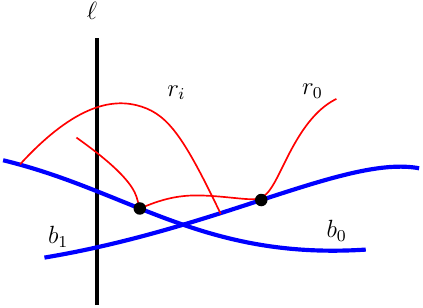}\label{fig:path-prop1}}
	\hspace{5mm}
	\subfloat[]{\includegraphics[width= 7cm]{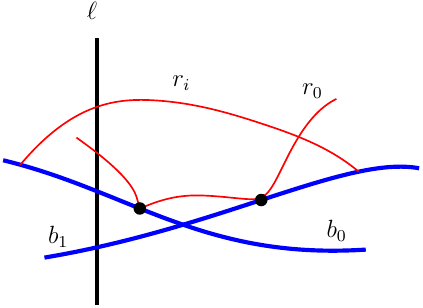}\label{fig:path-prop2}}
	\caption{Illustrations for the statement of Proposition~\ref{prop:structure}: $r_i$ intersects $\ell$ above $r_0$ and intersects $b_0(-,\ell)$, $r_0(b_0,+)$ and $b_1(b_0,+)$.}
	\label{fig:path-prop}		
\end{figure}
\end{prop}

\begin{proof}
We prove the claim by induction.
Consider the case $i=1$.
Before showing that $r_1$ satisfies the claim, we first look at $b_2$.
It intersects $\ell$ above $b_0$ and crosses $b_1$ to the right of $\ell$ by Propositions~\ref{prop:blue-cross} and~\ref{prop:blue-cross-between}.
It follows from Proposition~\ref{prop:b_0-b_2} that $b_0$ and $b_2$ intersect to the left of $\ell$ (refer to Figure~\ref{fig:path-prop-b2r1-a}).
\begin{figure}
	\centering
	\subfloat[If $I(b_1,r_1) <_x I(b_1,r_0)$ then $r_0$ cannot touch $b_0$.]{\includegraphics[width= 7cm]{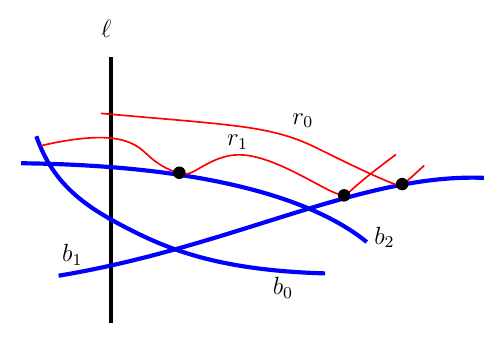}\label{fig:path-prop-b2r1-a}}
	\hspace{5mm}
	\subfloat[$I(b_1,r_0) <_x I(b_1,r_1)$. $r_1$ satisfies the properties of Proposition~\ref{prop:structure}]{\includegraphics[width= 7cm]{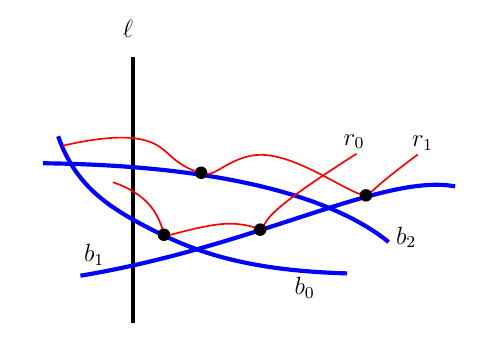}\label{fig:path-proof-b2r1-b}}
	\caption{Illustrations for the proof of Proposition~\ref{prop:structure}: the induction base.}
	\label{fig:path-proof-b2r1}		
\end{figure}
On the other hand the intersection of $b_1$ and $b_2$ is between
their touchings with $r_1$, hence to the right of $\ell$. Therefore, $b_2$ intersects $b_1(b_0,+)$.
Since $r_1$ is above $b_2$ it must also intersect $b_0$ to the left of $\ell$ and intersect $b_1(b_0,+)$.
It remains to show that $r_1$ intersects $\ell$ above $r_0$ and intersects $r_0(b_0,+)$.

Consider $r_0$ and note that it must touch $b_0(\ell,b_1)$ and (further to the right) touch $b_1(b_0,+)$. 
If $I(b_1,r_1) <_x I(b_1,r_0)$ then $I(r_0,r_1)$ is between these points by Proposition~\ref{prop:red-cross-between} and it follows that $r_0(\ell,r_1)$ is above $r_1(\ell,r_0)$ and hence $r_0$ cannot touch $b_0$ (see Figure~\ref{fig:path-prop-b2r1-a}).
Therefore, $I(b_1,r_0) <_x I(r_0,r_1) <_x I(b_1,r_1)$ which implies that $I(r_0,\ell) <_y I(r_1,\ell)$.
Since $r_0$ touches $b_0$ before touching $b_1$, this also implies that $r_1$ intersects $r_0(b_0,+)$.
Therefore, $r_1$ satisfies the properties above, see Figure~\ref{fig:path-proof-b2r1-b}
(note that $r_1$ cannot intersect $b_1$ to the left of $I(b_1,r_0)$ as in Figure~\ref{fig:path-prop}(a)).

\medskip
Suppose now that the claim holds for $r_i$, $i \ge 1$.
Observe that $b_{i+1}$ intersects $\ell$ above $b_1$ (as all the blue curves in $C$) and recall that $r_i$ and $r_{i+1}$ touch $b_{i+1}$ at Type~2 tangency points.
We consider two cases based on whether $b_{i+1}$ intersects $\ell$ above or below $b_0$.

\noindent\paragraph{Case 1:} $b_{i+1}$ intersects $\ell$ above $b_0$.
Since $L(b_{i+1}) <_x L(r_i) <_x R(r_i) <_x R(b_{i+1})$ it follows that $b_{i+1}$ must cross $b_0(r_i,\ell)$ and $b_1(b_0,r_i)$, see Figure~\ref{fig:cycle-proof6a}.
\begin{figure}[t]
	\centering
	\includegraphics[width= 8cm]{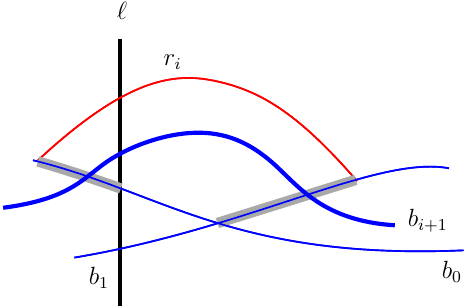}
	\caption{Proposition~\ref{prop:structure} Case 1: $I(b_0,\ell) <_y I(b_{i+1},\ell) <_y I(r_i,\ell)$. Since $L(b_{i+1}) <_x L(r_i) <_x R(r_i) <_x R(b_{i+1})$ it follows that $b_{i+1}$ must cross $b_0(r_i,\ell)$ and $b_1(b_0,r_i)$.}
	\label{fig:cycle-proof6a}		
\end{figure}
Since $r_{i+1}$ lies above $b_{i+1}$ it follows that, as $b_{i+1}$, it intersects $b_0$ to the left of $\ell$ and $b_1(b_0,+)$.
It remains to show that $I(r_0,\ell)  <_y I(r_{i+1},\ell)$ and that $r_{i+1}$ intersects $r_0(b_0,+)$.
We proceed by considering two subcases.

\medskip\noindent\underline{Case 1a}: $I(r_0,\ell) <_y I(b_{i+1},\ell)$.
Since $I(b_{i+1},\ell) <_y I(r_{i+1},\ell)$ it follows that $I(r_0,\ell) <_y I(r_{i+1},\ell)$.
Furthermore, $b_{i+1}$ must intersect $r_0(b_0,+)$ since $R(r_i) <_x R(b_{i+1})$ and by the induction hypothesis $r_0(b_0,+)$ intersects $r_i$.
Therefore $r_{i+1}$ intersects $r_0(b_0,+)$ as well (see Figure~\ref{fig:cycle-proof5}). 
\begin{figure}
	\centering
	\subfloat[]{\includegraphics[width= 7cm]{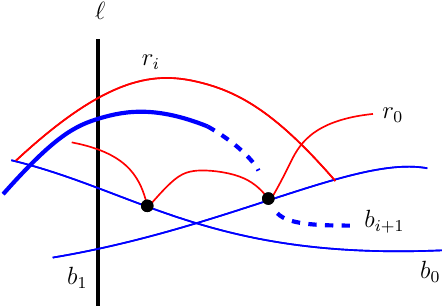}\label{fig:path-proof5b}}
	\hspace{5mm}
	\subfloat[]{\includegraphics[width= 7cm]{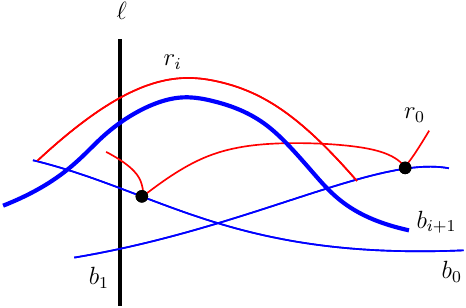}\label{fig:path-proof5c}}
	\caption{$I(r_0,\ell) <_y I(b_{i+1},\ell)$. Since $b_{i+1}$ has the desired properties so does $r_{i+1}$}
	\label{fig:cycle-proof5}		
\end{figure}

\medskip\noindent\underline{Case 1b}: $I(b_{i+1},\ell) <_y I(r_0,\ell)$.
Observe that $r_0$ touches $b_0$ at $b_0(\ell,b_1)$.
We have noted that $b_{i+1}$ intersects $b_0(r_i,l)$, hence $b_{i+1}$ must intersect $r_0(\ell,b_0)$, see Figure~\ref{fig:cycle-proof6}.
\begin{figure}[t]
	\centering
	\includegraphics[width= 8cm]{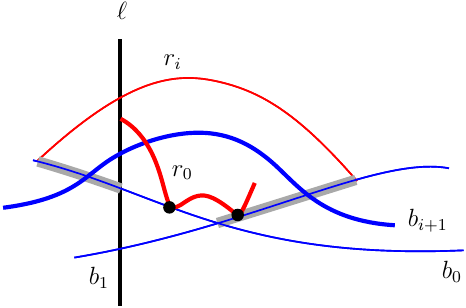}
	\caption{If $I(b_{i+1},\ell) <_y I(r_0,\ell)$ then $r_i$ does not intersect $r_0(b_0,+)$.}
	\label{fig:cycle-proof6}		
\end{figure}
Therefore $b_{i+1}(r_0,+)$ does not intersect $r_0(b_0,+)$.
Since $r_i$ lies above $b_{i+1}$ it follows that $r_i$ does not intersect $r_0(b_0,+)$ which is a contradiction.

\noindent\paragraph{Case 2:} $b_{i+1}$ intersects $\ell$ below $b_0$. Recall that $I(b_1,\ell)$ is the lowest among the intersection points of blue curves with $\ell$.
Since $I(b_0,b_1) <_x I(b_1,r_i) <_x R(b_{i+1})$ the curve $b_{i+1}$ must cross either $b_0(\ell,b_1)$ or $b_1(\ell,b_0)$.
In the latter case $b_{i+1}$ cannot intersect $b_0(b_1,+)$ by Proposition~\ref{prop:b-above}, therefore it must intersect  $b_0(-,\ell)$. 
But then $L(r_i) <_x L(b_{i+1})$ (see Figure~\ref{fig:cycle-proof4})
\begin{figure}[t]
	\centering
	\subfloat[If $b_{i+1}$ intersects $b_1(\ell,b_0)$ then $L(r_i) <_x L(b_{i+1})$.]{\includegraphics[width= 7cm]{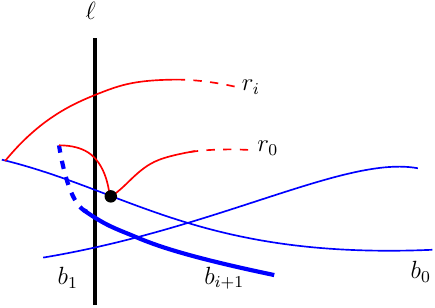}\label{fig:cycle-proof4}}
	\hspace{5mm}
	\subfloat[$b_{i+1}$ intersects $b_0(\ell,b_1)$. If $b_{i+1}$ crosses $r_0$, then it must cross it twice.]{\includegraphics[width= 7cm]{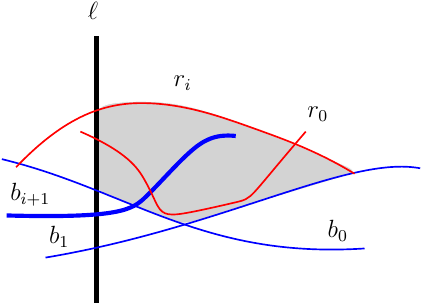}\label{fig:cycle-proof10}}
	\caption{$I(b_1,\ell) <_y I(b_{i+1},\ell) <_y I(b_0,\ell)$.}
	\label{fig:cycle-proof7}		
\end{figure}
Therefore, $b_{i+1}$ intersects $b_0(\ell,b_1)$.

Consider the region bounded by $\ell(r_i,b_0)$, $b_0(\ell,b_1)$, $b_1(b_0,r_i)$ and $r_i(\ell,b_1)$, see Figure~\ref{fig:cycle-proof10}.
Then $b_{i+1}$ enters this region at $b_0(\ell,b_1)$ and leaves it at $b_1(b_0,r_i)$.
Note that $b_{i+1}$ must intersect $r_0$ at this region since only within this region $b_{i+1}$ has a part above the upper envelope of $b_0$ and $b_1$ (where $r_0$ lies).
Furthermore, $b_{i+1}$ must touch $r_0$, for otherwise it must cross $r_0$ twice (see Figure~\ref{fig:cycle-proof10}).

It follows that $b_{i+1}$ crosses $b_0(r_0,b_1)$, then touches $r_0$ and then crosses $b_1(b_0,r_i)$.
One property that we wish to show is that $I(r_0,\ell) <_y I(r_{i+1},\ell)$.
Suppose that $I(r_{i+1},\ell) <_y I(r_0,\ell)$.
Since $r_{i+1}$ lies above $b_{i+1}$ it may intersect $b_1$ only at $b_1(b_{i+1},+)$.
It follows that $r_{i+1}$ crosses $r_0(\ell,b_{i+1})$ and intersects $b_1(r_0,+)$, see Figure~\ref{fig:cycle-proof8e}.
\begin{figure}[t]
	\centering
	\subfloat[If $I(r_{i+1},\ell) <_y I(r_0,\ell)$ then $(b_{i+1},r_0)$ is an edge in $G$ that creates a shorter cycle than $C$.]{\includegraphics[width= 7cm]{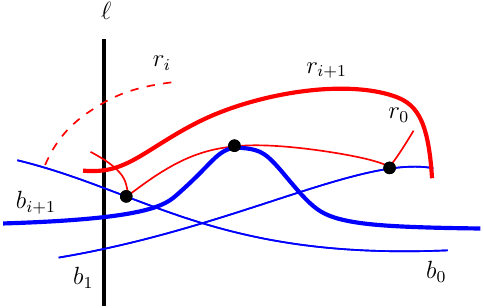}\label{fig:cycle-proof8e}}
	\hspace{5mm}
	\subfloat[If $I(r_{0},\ell) <_y I(r_{i+1},\ell)$ and $r_{i+1}$ intersects $r_0(-,b_{i+1})$ then $(b_{i+1},r_0)$ is an edge in $G$ that creates a shorter cycle than $C$.]{\includegraphics[width= 7cm]{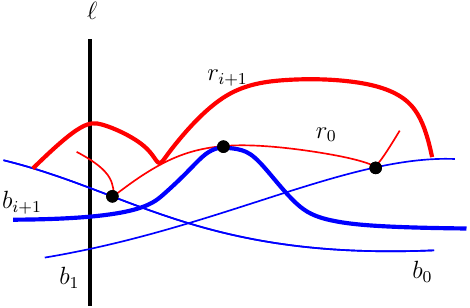}\label{fig:cycle-proof8g}}
	\hspace{5mm}
	\subfloat[If $I(r_{0},\ell) <_y I(r_{i+1},\ell)$ then $r_i$ has the desired properties.]{\includegraphics[width= 7cm]{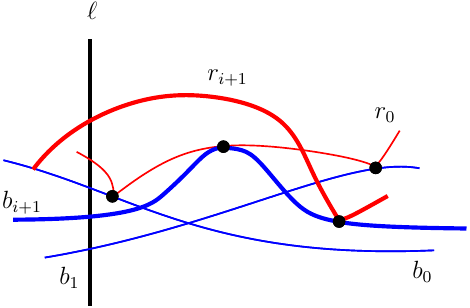}\label{fig:cycle-proof8f}}
	\caption{Concluding cases in the proof of Proposition~\ref{prop:structure}.}
	\label{fig:cycle-proof9}		
\end{figure}
However, this implies that $R(r_0) <_x R(r_{i+1}) <_x R(b_{i+1})$.
Furthermore, by induction $r_i$ intersects $r_0(b_0,+)$ and so it cannot intersect $r_0(-,b_0)$ and so $L(b_{i+1}) <_x L(r_i) <_x L(r_0)$, thus $I(b_{i+1},r_0)$ is a Type~2 tangency point.
Since $I(b_{i+1},r_0) <_x I(b_1,r_0)$, it also holds that $I(b_{i+1},r_0)$ is not the rightmost tangency point on $r_0$ and therefore $(b_{i+1},r_0)$ is an edge in $G$. But then $r_0-b_1-r_1-\ldots-b_{i+1}-r_0$ is a shorter cycle than $C$.

Thus $I(r_{0},\ell) <_y I(r_{i+1},\ell)$.
If $r_{i+1}$ intersects $r_0(-,b_{i+1})$ then it must touch it for otherwise $r_{i+1}$ cannot intersect $b_1(b_{i+1},+)$ (the only part of $b_1$ that lies above $b_{i+1}$ and may intersect $r_{i+1}$), see Figure~\ref{fig:cycle-proof8g}.
However, then $R(r_0) <_x R(r_{i+1}) <_x R(b_{i+1})$ which implies as before that $I(b_{i+1},r_0)$ is an unmarked Type~2 tangency point and there is a shorter cycle than $C$.
Therefore $r_{i+1}$ intersects $r_0(b_{i+1},+)$ (and hence, $r_0(b_0,+)$), $b_0(-,\ell)$ and $b_1(b_0,+)$, as desired (see Figure~\ref{fig:cycle-proof8f}).
\end{proof}

We now return to the proof of Lemma~\ref{lem:nested} and consider the cycle $C$.
It follows from Proposition~\ref{prop:structure} that $r_k$ intersects $b_0$ to the left of $\ell$ and therefore $(b_0,r_k)$ cannot be an edge in $G$.
Thus $G$ is a forest and has at most $n-1$ edges.
This implies that there are at most $2n-1$ Type~2 tangency points to the right of $\ell$ and at most $8n-4$ tangency points of Types~1 and~2.
\end{proof}

\subsection{Bounding touching pairs of Type~3 or~4}

We first describe the main idea before going into details:
By symmetry it is enough to consider just tangency points of Type~4 which are to the right of a vertical line that intersects all the curves in $\cC$.
We will prove that there are linearly many such tangencies by showing that after some additional pruning we can order the edges of the tangencies graph such that there is no monotone increasing path of length $7$.
A linear bound on the size of this graph then follows by the next result about edge-ordered graphs, attributed to R\"odl~\cite{rodl} in~\cite{GERBNER202366}.
For completeness and since the latter reference is not easily accessible we reprove it. 

\begin{lem}~\cite{rodl}\label{lem:k-path}
	Let $G=(V,E)$ be an $n$-vertex graph and let $<$ be a total order of its edges.
	Let $k$ be an integer such that $G$ has at least ${k+1\choose 2}n$ edges. 
	Then $G$ contains a monotone increasing path of $k$ edges, that is, a path $e_1-e_2-\ldots-e_k$ such that $e_1 < e_2 < \ldots < e_k$.
\end{lem}

\begin{proof}
We prove by double induction.
For any $n$ and $k=1$ the claim trivially holds as well as for every $k$ and $n=1$.
For the induction step, we assume that the claim holds for $n$ and $k-1$ and also for $k$ and $n-1$ and consider an $n$-vertex graph with at least ${k+1\choose 2}n$ totally ordered edges.
If the graph contains a vertex of degree smaller than $k$, then we remove this vertex and obtain a graph with $n-1$ vertices and at least ${k+1\choose 2}n-(k-1) \ge {k+1\choose 2}(n-1)$ edges.
Thus, by the induction hypothesis a monotone increasing path of length $k$ exists.
Otherwise, if the degree of every vertex is at least $k$, then we remove for every vertex the $k$ highest incident edges (with respect to $<$).
There are at least ${k+1\choose 2}n-nk \ge {k\choose 2}n$ remaining edges, therefore by the induction hypothesis the graph that we get has a monotone increasing path of $k-1$ edges and $k$ vertices. 
Denote this path by $p$ and let $v$ be the last vertex on $p$.
Then at least one of the $k$ removed edges at $v$ has an endpoint which is not on $p$. By adding this edge to $p$ we get a monotone increasing path of $k$ edges.
\end{proof}

\begin{lem}\label{lem:non-nested}
There are at most $1152n$ tangency points of Type~3 or~4.
\end{lem}

\begin{proof}
Since all the curves in $\cC$ are pairwise intersecting and $x$-monotone there is a vertical line $\ell$ that intersects all of them.
By slightly shifting $\ell$ if needed we may assume that no two curves intersect $\ell$ at the same point.
We assume without loss of generality that at least half of all the tangency points of Type~3 or~4 are to the right of $\ell$, for otherwise we may reflect all the curves about $\ell$.
We may further assume that at least half of the tangency points of Type~3 or~4 to the right of $\ell$ are of Type~4, for otherwise we may reflect all the curves about the $x$-axis.
Henceforth, we consider only Type~4 tangency points to the right of $\ell$.

By Proposition~\ref{prop:poset} a curve cannot touch one curve from above and another curve from below at Type~4 tangency points. Thus, we may partition the curves into \emph{blue} curves and \emph{red} curves such that at every tangency point a blue curve touches a red curve from below (we ignore curves that contain no tangency points among the ones that we consider).

Clearly, there are no Type~4 tangencies among the blue curves, however, there might be tangencies of other types among them.
Next we wish to obtain a subset of the blue curves such that every pair of them are crossing and they together contain a percentage of the tangency points that we consider.
It follows from Proposition~\ref{prop:poset} that the largest chain in the partially ordered set of the blue curves with respect to $\prec_1$ is of length two. Therefore, by Mirsky's Theorem (the dual of Dilworth's Theorem) the blue curves can be partitioned into two antichains with respect to $\prec_1$.
The blue curves of one of these antichains contain at least half of the tangency points that we consider.
By continuing with this set of blue curves and applying the same argument twice more with respect to $\prec_2$ and $\prec_3$ we obtain a set of pairwise crossing blue curves that together contain at least $1/8$ of the tangency points of Type~4 to the right of $\ell$.
Henceforth we consider these blue curves and the red curves that touch at least one of them at a Type~4 tangency point to the right of $\ell$.

Let $G=(B\cup R,E)$ be the (bipartite) \emph{tangencies graph} of these blue and red curves.
That is, $B$ corresponds to the blue curves, $R$ corresponds to the red curves and $E$ corresponds to pairs of touching blue and red curves (at Type~4 tangency points to the right of $\ell$).
We order the edges of $G$ according to the order of their corresponding tangency points from left to right.

The lemma follows from Lemma~\ref{lem:k-path} and the next claim.

\begin{prop}\label{prop:no-7-path}
$G$ does not contain a monotone increasing path of $7$ edges starting at $B$.
\end{prop}

\begin{proof}
Suppose that $G$ contains a monotone increasing path $b_1-r_1-\ldots-b_4-r_4$, such that $b_i \in B$ and $r_i \in R$, for $i=1,2,3,4$.
Since all the curves intersect $\ell$ and $R(b_i) <_x R(r_{i-1})$, we have:

\begin{obs}\label{obs:reds-cross-non-nested}
For $i=1,2,3$ we have that $I(b_{i+1},r_i) <_x I(r_i,r_{i+1}) < I(b_{i+1},r_{i+1})$ and $r_{i+1}(-,r_i)$ lies above $r_i(-,r_{i+1})$ (see Figure~\ref{fig:increasing-path}).
\begin{figure}[t]
	\centering
	\includegraphics[width= 12cm]{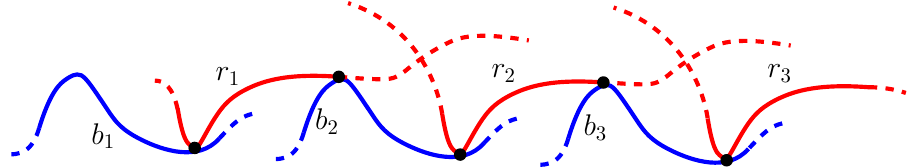}
	\caption{Since the curves intersect $\ell$ and $R(b_i) <_x R(r_{i-1})$, it follows that $I(b_2,r_1) <_x I(r_1,r_2) < I(b_2,r_2)$ and $I(b_3,r_2) <_x I(r_2,r_3) < I(b_3,r_3)$.}
	\label{fig:increasing-path}		
\end{figure}
\end{obs}

Considering consecutive blue curves in the path, observe that $I(b_i,b_{i+1})$ cannot be to the right of $I(b_{i+1},r_i)$, since in such a case $b_{i+1}$ must intersect $b_i$ or $r_i$ twice to be able to intersect $\ell$, see Figure~\ref{fig:b1r1b2-b}.
\begin{figure}
	\centering
	\includegraphics[width=7cm]{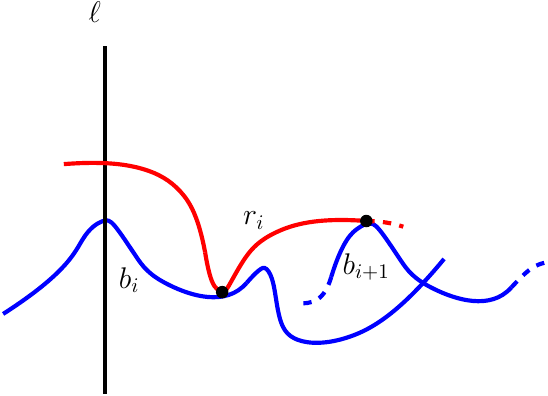}
	\caption{$I(b_i,b_{i+1})$ cannot be to the right of $I(b_{i+1},r_i)$.}
	\label{fig:b1r1b2-b}		
\end{figure}

\begin{obs}\label{obs:blues-cross-non-nested}
For $i=1,2,3$ we have that $I(b_i,b_{i+1}) <_x I(b_{i+1},r_i)$.
\end{obs}

Thus, $I(b_1,b_2)$ is to the left of $I(b_2,r_1)$. We consider two cases based on its location with respect to $I(b_1,r_1)$.

\noindent\paragraph{Case 1:} $I(b_1,b_2) <_x I(b_1,r_1)$.
This implies that $R(b_1) <_x I(b_2,r_1)$. 
Since $I(b_2,r_1) <_x I(r_1,r_2)$ and $r_2(-,r_1)$ lies above $r_1(-,r_2)$ which lies above $b_1$, it follows that $r_2$ and $b_1$ may intersect only to the left of $L(r_1)$. 
We must also have $I(b_1,b_2) <_x I(b_1,r_2)$ since $L(b_2)<L(r_2)$ and $r_2$ lies above $b_2$ which must be above $b_1$ to the left of $I(b_1,b_2)$ (see Figure~\ref{fig:b1r1b2r2b3-b}). 
\begin{figure}[t]
	\centering
	\subfloat[If $b_3$ intersects $b_2(r_2,+)$ then it cannot intersect $r_1$.]{\includegraphics[width=6cm]{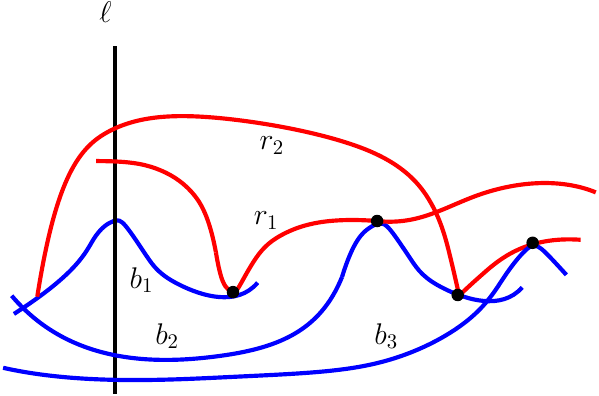}\label{fig:b1r1b2r2b3-b}}
	\hspace{5mm}
	\subfloat[If $b_3$ intersects $b_2(-,r_2)$ then $b_2$ and $r_3$ do not intersect.]{\includegraphics[width=7.5cm]{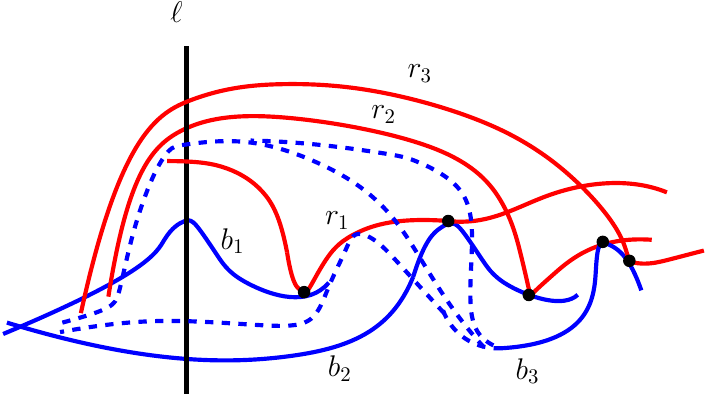}\label{fig:b1r1b2r2b3-c}}
	\caption{Case 1: If $I(b_1,b_2) <_x I(b_1,r_1)$ then $I(b_1,b_2) <_x I(b_1,r_2) <_x L(r_1)$. }
	\label{fig:b1r1b2r2b3}		
\end{figure}

Considering $b_3$, we observe that it cannot intersect $b_2(r_2,+)$ since then it 
does not intersect $r_1$.
Indeed, suppose that $b_3$ intersects $b_2(r_2,+)$ and refer to Figure~\ref{fig:b1r1b2r2b3-b}. $b_3$ lies below $r_2$ and $R(b_3) <_x R(r_2)$, therefore $b_3$ may not intersect $r_1(r_2,+)$ which lies above $r_2$.  
Since $I(r_1,r_2) <_x I(b_2,r_2) <_x I(b_2,b_3)$ it follows that $b_3(b_2,+)$ cannot intersect $r_1$.
The other part of $b_3$, $b_3(-,b_2)$, lies below $b_2$ which lies below $r_1$ and has its left endpoint to the left of $L(r_1)$. Therefore $b_3(-,b_2)$ cannot intersect $r_1$ as well. 

Therefore, $b_3$ crosses $b_2(-,r_2)$.
This implies that $R(b_2) <_x I(b_3,r_2)$, see Figure~\ref{fig:b1r1b2r2b3-c}.
We claim that $b_3$ and $r_2$ `block' $r_3$ from intersecting $b_2$.
Indeed, since $R(b_2) <_x I(b_3,r_2) <_x I(r_2,r_3)$ and $r_3(-,r_2)$ lies above $r_2(-,r_3)$ which lies above $b_2$ it follows that $r_3$ may intersect $b_2$ only to the left of $L(r_2)$.
However, $r_2(r_1,+)$ `blocks' $b_3$ from intersecting $r_1$ to the right of $I(r_1,r_2)$, therefore $b_3$ must intersect $r_1(-,r_2)$, which implies that $b_3$ must cross $b_2$ to the right of $L(r_1)$ which is to the right of $L(r_2)$, see Figure~\ref{fig:b1r1b2r2b3-c}.
But then $b_3$ is above $b_2$ to the left of $L(r_2)$, and since $L(b_3) <_x L(r_2)$ and $L(b_3) <_x L(r_3)$ it follows that $b_3$ `blocks' $b_2$ from intersecting $r_3$ to the left of $L(r_2)$.
Therefore, $b_2$ and $r_3$ do not intersect.

This concludes Case~1. Note that we have not used the existence of $b_4$ and $r_4$, that is, we only considered the path $b_1-r_1-b_2-r_2-b_3-r_3$ in $G$.

\noindent\paragraph{Case 2:} $I(b_1,r_1) <_x I(b_1,b_2) <_x I(b_2,r_1)$.
We claim that $I(b_2,b_3) <_x I(b_2,r_2)$. 
Indeed, suppose that $I(b_2,r_2) <_x I(b_2,b_3)$ and refer to Figure~\ref{fig:b1r1b2r2b3-e}.
\begin{figure}
	\centering
	\includegraphics[width=7cm]{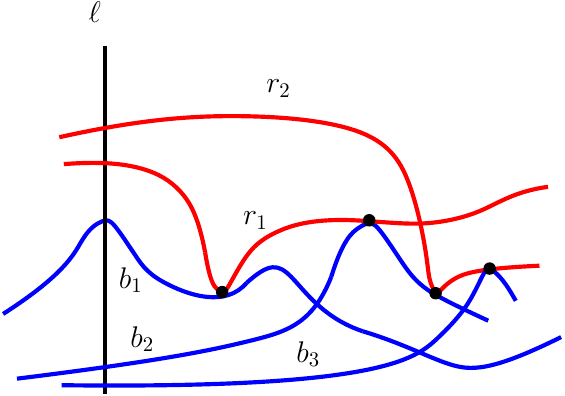}
	\caption{Case 2:  $I(b_1,r_1) <_x I(b_1,b_2) <_x I(b_2,r_1)$. If $I(b_2,r_2) <_x I(b_2,b_3)$ then $b_3$ and $r_1$ do not intersect.}
	\label{fig:b1r1b2r2b3-e}		
\end{figure}
Then $b_3(-,b_2)$ must lie below $b_2(-,b_3)$.
Since $b_2$ lies below $r_1$ and $L(b_2) <_x L(r_1)$ it follows that $b_3(-,b_2)$ cannot intersect $r_1$.
Considering the other part of $b_3$, $b_3(b_2,+)$, it lies below $r_2$ and its right endpoint is to the left of $R(r_2)$.
Since $r_2$ is below $r_1$ to the right of $I(r_1,r_2)$ and $I(r_1,r_2) <_x I(b_2,r_2) <_x I(b_2,b_3)$, it follows that $b_3(b_2,+)$ cannot intersect $r_1$.

Therefore $I(b_2,b_3) <_x I(b_2,r_2)$. 
However, then we are in Case~1 for the path $b_2-r_2-b_3-r_3-b_4-r_4$, which is impossible.
\end{proof}

Returning to the proof of Lemma~\ref{lem:non-nested} by Proposition~\ref{prop:no-7-path} $G$ has no monotone increasing path of length $8$ (starting with an edge of any color) and then by Lemma~\ref{lem:k-path} $G$ has less than $36n$ edges.
This in turn implies that there are at most $8 \cdot 2 \cdot 2 \cdot 36n = 1152n$ tangency points of Types~3 and~4 together (in fact at most $1152n-32$).
\end{proof}

By Lemmata~\ref{lem:nested} and~\ref{lem:non-nested} there are at most $1160n-4$ tangency points among the curves in $\cC$. This concludes the proof of Theorem~\ref{thm:pw-x-monotone-1-intersecting}.

\section{Discussion}
\label{sec:Discussion}

We have shown that $n$ $x$-monotone pairwise intersecting $1$-intersecting curves determine $O(n)$ tangencies. The constant hiding in the big-$O$ notation is rather large, since, for simplicity, we did not make much of an effort to get a smaller constant. In particular, our upper bound can be improved by considering more cases. For example, in the proof of Lemma~\ref{lem:non-nested} we may consider tangencies among blue curves and avoid using the dual of Dilworth's Theorem. It is also enough to forbid a monotone increasing path of $5$ edges in Proposition~\ref{prop:no-7-path}, again by considering more cases.
It would be interesting to determine the exact maximum number of tangencies among a set of $n$ $x$-monotone curves each two of which intersect at exactly one point.
The best lower bound we have is $\lfloor 3.5n \rfloor -8$, see Figure~\ref{fig:lbconstr}.
\begin{figure}
	\centering
	\includegraphics[width= 14cm]{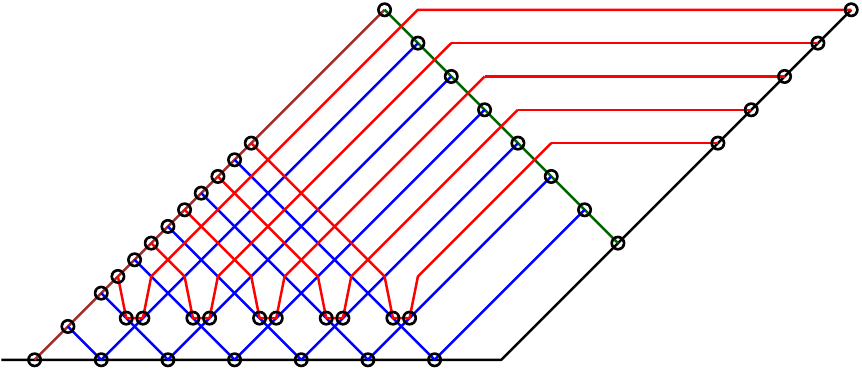}
	\caption{$n$ $x$-monotone pairwise intersecting $1$-intersecting curves can determine $\lfloor 3.5n\rfloor -8$ tangencies (one can slightly elongate the curves in a way that there is no tangency at an endpoint).}
	\label{fig:lbconstr}		
\end{figure}

Suppose that we allow more than two curves to intersect at a single point but count the number of \emph{tangency points} rather than the number of tangent pairs of curves. Is it still true that there are linearly many tangency points? For $n$ $1$-intersecting curves which are not necessarily $x$-monotone one can get as many as $\Omega(n^{4/3})$ tangency points via the construction of that many point-line incidences, see Figure~\ref{fig:many-tangency-points} for an illustration.
\begin{figure}[t]
	\centering
	\includegraphics[width= 12cm]{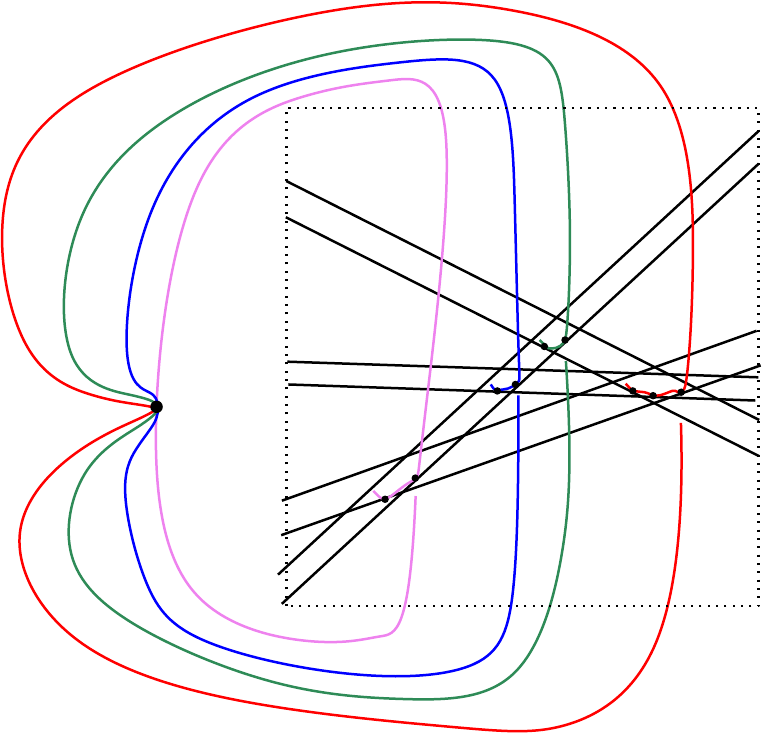}
	\caption{$n$ pairwise intersecting $1$-intersecting curves might determine $\Omega(n^{4/3})$ tangency points.}
	\label{fig:many-tangency-points}		
\end{figure}

\paragraph{Acknowledgments.} We thank anonymous reviewers for their helpful suggestions for improving the presentation of the paper. The lower-bound construction shown in Figure~\ref{fig:lbconstr} is due to one of them.

\footnotesize
\bibliographystyle{plainurl}
\bibliography{xmon}

@Incollection{Ackerman2013,
	author="Ackerman, Eyal",
	editor="Pach, J{\'a}nos",
	title="The Maximum Number of Tangencies Among Convex Regions with a Triangle-Free Intersection Graph",
	bookTitle="Thirty Essays on Geometric Graph Theory",
	year="2013",
	publisher="Springer New York",
	address="New York, NY",
	pages="19--30",
	crossref="pach2012thirty",
	isbn="978-1-4614-0110-0",
	doi="10.1007/978-1-4614-0110-0_3",
	url="https://doi.org/10.1007/978-1-4614-0110-0_3"
}

@mastersthesis{rodl,
	author  = "V. R{\"o}dl",
	school  = "Charles University",
	year    = "1973"
}

@article{GERBNER202366,
	title = {Tur\'an problems for edge-ordered graphs},
	journal = {Journal of Combinatorial Theory, Series B},
	volume = {160},
	pages = {66-113},
	year = {2023},
	issn = {0095-8956},
	doi = {https://doi.org/10.1016/j.jctb.2022.12.006},
	url = {https://www.sciencedirect.com/science/article/pii/S0095895622001265},
	author = {D\'aniel Gerbner and Abhishek Methuku and D\'aniel T. Nagy and D{\"o}m{\"o}t{\"o}r P{\'a}lv{\"o}lgyi and G\'abor Tardos and M\'at\'e Vizer}
}

@book{pach2012thirty,
	title={Thirty Essays on Geometric Graph Theory},
	editor={Pach, J.},
	isbn={9781461401100},
	year={2012},
	publisher={Springer New York}
}

@article{Treml,
	author    = {J{\'{a}}nos Pach and
	Andrew Suk and
	Miroslav Treml},
	title     = {Tangencies between families of disjoint regions in the plane},
	journal   = {Comput. Geom.},
	volume    = {45},
	number    = {3},
	pages     = {131--138},
	year      = {2012},
	url       = {https://doi.org/10.1016/j.comgeo.2011.10.002},
	doi       = {10.1016/j.comgeo.2011.10.002},
	bibsource = {dblp computer science bibliography, https://dblp.org}
}

@article{GYORGYI201829,
	title = "On the number of touching pairs in a set of planar curves",
	journal = "Computational Geometry",
	volume = "67",
	pages = "29--37",
	year = "2018",
	issn = "0925-7721",
	doi = "https://doi.org/10.1016/j.comgeo.2017.10.004",
	url = "http://www.sciencedirect.com/science/article/pii/S0925772117300949",
	author = "P\'eter Gy{\"o}rgyi and B\'alint Hujter and S\'andor Kisfaludi-Bak"
}

@article{PS91,
	author    = {J{\'{a}}nos Pach and
	Micha Sharir},
	title     = {On Vertical Visibility in Arrangements of Segments and the Queue Size
	in the {B}entley-{O}ttmann Line Sweeping Algorithm},
	journal   = {{SIAM} J. Comput.},
	volume    = {20},
	number    = {3},
	pages     = {460--470},
	year      = {1991},
	url       = {https://doi.org/10.1137/0220029},
	doi       = {10.1137/0220029},
	bibsource = {dblp computer science bibliography, https://dblp.org}
}

@article{erdos,
	ISSN = {00029890, 19300972},
	URL = {http://www.jstor.org/stable/2305092},
	author = {P. Erd\H{o}s},
	journal = {The American Mathematical Monthly},
	number = {5},
	pages = {248--250},
	publisher = {Mathematical Association of America},
	title = {On Sets of Distances of $n$ Points},
	urldate = {2023-03-15},
	volume = {53},
	year = {1946}
}

@Article{agarwal,
  author    = {Agarwal, Pankaj K and Nevo, Eran and Pach, J{\'a}nos and Pinchasi, Rom and Sharir, Micha and Smorodinsky, Shakhar},
  title     = {Lenses in arrangements of pseudo-circles and their applications},
  journal   = {Journal of the ACM (JACM)},
  year      = {2004},
  volume    = {51},
  number    = {2},
  pages     = {139--186},
  publisher = {ACM},
}

@misc{pachpc,
  author = "Pach, J\'anos",  
  howpublished = "personal communication"
}

@article{RT95, 
	title={Intersections of curve systems and the crossing number of ${C}_5 \times {C}_5$}, 
	volume={13}, 
	DOI={10.1007/BF02574034}, 
	number={2}, 
	journal={Discrete Comput. Geom.}, 
	author={R.B.\ Richter and C.\ Thomassen}, 
	year={1995}, 
	pages={149--159}
}

@article{PRT16, 
	title={On the {R}ichter-{T}homassen Conjecture about Pairwise Intersecting Closed Curves}, 
	volume={25}, 
	DOI={10.1017/S0963548316000043}, 
	number={6},
	journal = {Combinatorics, Probability and Computing},
	publisher={Cambridge University Press}, 
	author={Pach, J\'anos and Rubin, Natan and Tardos, G\'abor}, 
	year={2016}, 
	pages={941--958}
}

@article{PRT18,
	title = {A Crossing Lemma for {J}ordan curves},
	journal = {Advances in Mathematics},
	volume = {331},
	pages = {908--940},
	year = {2018},
	issn = {0001-8708},
	doi = {https://doi.org/10.1016/j.aim.2018.03.015},
	url = {https://www.sciencedirect.com/science/article/pii/S0001870818301026},
	author = {J\'anos Pach and Natan Rubin and G\'abor Tardos}
}

@misc{KP21,
	doi = {10.48550/ARXIV.2111.08787},	
	url = {https://arxiv.org/abs/2111.08787},	
	author = {Keszegh, Bal{\'a}zs and P{\'a}lv{\"o}lgyi, D{\"o}m{\"o}t{\"o}r},	
	title = {The number of tangencies between two families of curves},	
	publisher = {arXiv},
	year = {2021},	
	copyright = {arXiv.org perpetual, non-exclusive license}
}

@inbook{pachbook,
	publisher = {John Wiley and Sons Ltd},
	isbn = {9781118033203},
	title = {Combinatorial Geometry},
	author={Pach, J\'anos and Agarwal, Pankaj K.},
	chapter = {11},
	pages = {177--178},
	doi = {https://doi.org/10.1002/9781118033203.ch11},
	year = {1995}
}

@article{S99, 
	title={On the {R}ichter-{T}homassen Conjecture about Pairwise Intersecting Closed Curves}, 
	volume={75}, 
	DOI={10.1006/jctb.1998.1858}, 
	number={1},
	journal = {J. Comb. Theory, Ser. B},
	author={Gelasio Salazar}, 
	year={1999}, 
	pages={56--60}
}

@article{maya,
	author    = {Maya Bechler{-}Speicher},
	title     = {A Crossing Lemma for Families of Jordan Curves with a Bounded Intersection
	Number},
	journal   = {CoRR},
	volume    = {abs/1911.07287},
	year      = {2019},
	url       = {http://arxiv.org/abs/1911.07287},
	eprinttype = {arXiv},
	eprint    = {1911.07287},
	bibsource = {dblp computer science bibliography, https://dblp.org}
}

@article{AKP23,
title = {On tangencies among planar curves with an application to coloring {L}-shapes},
journal = {European Journal of Combinatorics},
pages = {103837},
year = {2023},
issn = {0195-6698},
doi = {https://doi.org/10.1016/j.ejc.2023.103837},
url = {https://www.sciencedirect.com/science/article/pii/S0195669823001555},
author="Ackerman, Eyal
	and Keszegh, Bal{\'a}zs
	and P{\'a}lv{\"o}lgyi, D{\"o}m{\"o}t{\"o}r"
}

\end{document}